\documentclass[11pt, a4paper]{article}
\usepackage{fullpage}

\usepackage[utf8]{inputenc}

\usepackage{amsfonts}
\usepackage{amsmath} 
\usepackage{amssymb}
\usepackage{amsthm}
\usepackage{verbatim}
\usepackage{mathtools}
\usepackage{color}   
\usepackage[dvipsnames]{xcolor}
\usepackage[hypertexnames=false]{hyperref}
\hypersetup{
    colorlinks=true, 
    linkcolor=blue, 
    urlcolor=blue, 
    citecolor= ForestGreen,
    linktoc=all
}
\usepackage{cleveref}
\crefname{enumi}{part}{parts}
\usepackage{authblk}

\newtheorem{theorem}{Theorem}[section]
\newtheorem{lemma}[theorem]{Lemma}
\newtheorem{proposition}[theorem]{Proposition}
\theoremstyle{definition}
\newtheorem{definition}[theorem]{Definition} 
\newtheorem{notation}[theorem]{Notation}

\DeclarePairedDelimiter\ceil{\lceil}{\rceil}
\DeclarePairedDelimiter\floor{\lfloor}{\rfloor}

\newcommand{\Sym}[0]{\mathrm{Sym}}

\newcommand{\Aut}[0]{\mathrm{Aut}}

\newcommand{\sym}[2]{\mathrm{S}_{#1,#2}}
\newcommand{\alt}[2]{\mathrm{A}_{#1,#2}}
\newcommand{\ksubsets}[2]{[#2]_#1}

\newcommand{\soc}[1]{\mathrm{soc}(#1)}
\newcommand{\normOfSoc}[0]{M}

\newcommand{\Sy}[0]{\mathrm{S}}
\newcommand{\Al}[0]{\mathrm{A}}
\newcommand{\Amkl}[0]{A(m,k,l)}

\title{Primitive normalisers in quasipolynomial time}
\date{\today}

\author[1]{Mun See Chang}
\author[2]{Colva M.\ Roney-Dougal}
\affil[1]{School of Computer Science, University of St Andrews}
\affil[2]{School of Mathematics and Statistics, University of St Andrews}
\affil[ ]{\texttt {\{msc2,Colva.Roney-Dougal\}@st-andrews.ac.uk}}

\begin{document}

\maketitle

\begin{abstract}
The normaliser problem has as input two subgroups $H$ and $K$ of the symmetric group $\Sy_n$, and asks for a generating set for $N_K(H)$: it is not known to have a subexponential time solution.  It is proved in \cite{normOfPrim} that if $H$ is primitive then the normaliser problem can be solved in quasipolynomial time. We show that for all subgroups $H$ and $K$ of $\Sy_n$, in quasipolynomial time we can decide whether $N_{\Sy_n}(H)$ is primitive, and if so compute $N_K(H)$. Hence we reduce the question of whether one can solve the normaliser problem in quasipolynomial time to the case where the normaliser is known not to be primitive.
\end{abstract}


\section{Introduction}
\label{section: intro}

The \emph{normaliser problem} asks for a generating set for $N_K(H)$, given subgroups $K$ and $H$ of $\Sy_n$. 
It is shown in \cite{wiebking} that the problem can be solved in simply exponential time $2^{O(n)}$, but there is no known subexponential solution to the general problem. 
A permutation group problem $\mathcal{P}$ is said to be \emph{quasipolynomial} if there exists a constant $c$ such that $\mathcal{P}$ can be solved in time $2^{O(\log^c{n})}$, where $n$ is the degree of the underlying group or groups. 

It is shown in \cite{normOfPrim} that if $H$ is primitive, then the normaliser problem is quasipolynomial. In this paper, we will show that if $N_{\Sy_n}(H)$ is primitive then the normaliser problem is quasipolynomial. Our main theorem is the following. 

\begin{theorem} \label{norm is prim in poly time} \label{main theorem}
Let subgroups $H = \langle X \rangle$ and $K = \langle Y \rangle$ of $\Sy_n$ be given.  
\begin{enumerate}
    \item \label{decide if norm in prim} We can decide if $N=N_{\Sy_n}(H)$ is primitive, and if so construct $N$, in time $2^{O(\log^3{n})}$.  
    \item \label{if norm prim then compute norm} If $N$ is primitive, then we can compute $N_K(H)$ in time $2^{O(\log^3{n})}$. 
\end{enumerate}
\end{theorem}

\noindent (Throughout the paper we shall assume all generating sets have size at most $n$: see \Cref{poly time lib}.\ref{replace by small gen set}).

In fact, we can compute $N_{K}(H)$ in time $2^{O(\log^3{n})}$ except when: (i) $H$ is intransitive; or (ii) $|H|\geq n^{1+\floor{\log{n}}}$ but $H$ is not ample (see \Cref{def: ample group}); or (iii) $|H|<n^{1+\floor{\log{n}}}$ but $H$ does not have a small base or a small generating set (see \Cref{get small gen set and base for $H$}). 
In this latter case, we can still compute $N_{K}(H)$ in quasipolynomial time $2^{O(\log^5{n})}$, see \Cref{small group's norm in quasipoly time}.

Babai in \cite{babaiGI} gave a $2^{O(\log^c{n})}$ time solution to the string isomorphism problem, and Helfgott in \cite{helfgottGI} showed that we can take $c=3$. 
The setwise stabiliser problem is a special case of the string isomorphism problem, and was shown in \cite{luksHierarchy} to be polynomial-time equivalent to the intersection problem. 
Hence to show that $N_{K}(H) = N \cap K$ can be computed in time $2^{O(\log^3{n})}$, it suffices to prove that $N$ can be computed in time $2^{O(\log^3{n})}$.

In \Cref{section: prelim}, we first present some preliminaries on permutation groups and permutation group algorithms. We then see how we can determine that certain groups $H$ have base and generating set of size $O(\log{n})$ in quasipolynomial time and prove \Cref{small group's norm in quasipoly time}. 
In \Cref{section: ample}, we introduce the class of ample groups and show that if $H$ is ample then $N = N_{\Sy_n}(H)$ can be computed in quasipolynomial time. 
Finally these results come together to prove \Cref{main theorem}.


\section{Preliminaries and small groups}
\label{section: prelim} \label{section: small}

This section first collects background on permutation groups and polynomial time computation and then studies small groups.   

Let $G \leq \Sym(\Omega)$ and $H \leq \Sym(\Gamma)$. Then $G$ and $H$ are \emph{permutation isomorphic} if there exist an isomorphism $\phi: G \rightarrow H$ and a bijection $\sigma : \Omega \rightarrow \Gamma$ such that $\sigma(\omega^g)= \sigma(\omega)^{\phi(g)}$ for all $\omega \in \Omega$ and $g \in G$. 
We say that such a pair $(\phi,\sigma)$ is a \emph{permutation isomorphism} from $G$ to $H$. 

\begin{notation}
Let $\ksubsets{k}{m}$ denote the set of all $k$-subsets of $\{1,2, \ldots, m\}$ with $1 \leq k \leq m/2$. 
Let $\alt{m}{k}$ and $\sym{m}{k}$ denote $\Al_m$ and $\Sy_m$, acting on $\ksubsets{k}{m}$. 
Let $\ksubsets{k}{m}^l$ denote the set of all $l$-tuples of $\ksubsets{k}{m}$, and let $\Amkl$ be the group $(\alt{m}{k})^l$ acting coordinatewise on $\ksubsets{k}{m}^l$. 
\end{notation}

We will be using the following key result, proved by Mar\'{o}ti. 

\begin{theorem} [{\cite{marotiPrimOrders}}]\label{big small and sporadic primitive groups}
Let $G$ be a primitive subgroup of $\Sy_n$. 
Then at least one of the following holds. 
\begin{enumerate}
    \item $G$ is $\mathrm{M}_{11}$, $\mathrm{M}_{12}$, $\mathrm{M}_{23}$ or $\mathrm{M}_{24}$ with their 4-transitive actions. 
    \item There exist $m \geq 5$, $1 \leq k < m/2$ and $l\geq1$ such that, up to permutation isomorphism, $\Amkl \trianglelefteq G \leq \sym{m}{k} \wr \Sy_l$.
    \item $|G| < n^{1+ \floor{\log{n}}}$. 
\end{enumerate}
\end{theorem}
We shall call these classes \emph{Mathieu}, \emph{large} and \emph{small} respectively. 
A primitive group is of \emph{type PA} if it is in product action and the component of the base group is almost simple (see \cite{ONanScottForPrim}).  
It follows that a large primitive group is either almost simple (when $l=1$) or of type PA (when $l>1)$. 

For $G = \langle z_1, z_2, \ldots, z_k \rangle \leq \Sy_n$ and $L = \langle y_1, y_2, \ldots, y_k \rangle \leq \Sy_m$, a homomorphism $\phi:G\rightarrow L$ is \emph{given by generator images} if it is encoded by a list $[z_1, \ldots, z_k, y_1, \ldots, y_k, \phi(z_1), \ldots, \phi(z_k)]$.
We shall assume that all homomorphisms are given by generator images, that we have a library of standard representations of all finite simple groups, and that their automorphism groups are known. 

The following results are standard (see, for example, \cite{handbookCGT}). 

\begin{lemma} \label{poly time lib}
Given $G = \langle Z \rangle \leq \Sy_n$, the following can be done in time polynomial in $|Z| \cdot n$. 
\begin{enumerate}
    \item \label{replace by small gen set}  \label{membership testing} \label{compute order} \label{compute orbits} \label{compute pt stab}  \label{get a non-redundant base} \label{check if primitive} Replace $Z$ by a generating set for $G$ of size at most $n$; 
    given $\sigma \in \Sy_n$, decide if $\sigma \in G$; 
    compute $|G|$; compute the orbits of $G$; compute the stabiliser in $G$ of any given point; compute an irredundant base for $G$; decide if $G$ is primitive.  
    \item \label{check if map is homom} \label{compute im and preim of homom} 
    Given a map $\phi: G \rightarrow L$ by the images of $Z$, decide if $\phi$ extends to an isomorphism; 
    given an isomorphism $\phi:G \rightarrow L$, compute $\phi^{-1}$.  
    \item \label{find a minimal normal subgroup} \label{compute cent of normal subgroup} \label{compute socle} 
    Find a minimal normal subgroup of $G$; 
    compute $C_G(J)$ for $J \trianglelefteq G$; 
    find generators for the socle $\soc{G}$.
    \item \label{compute compo series} \label{decide if simple} Compute the composition factors of $G$; decide if $G$ is simple and if so, give an isomorphism from $G$ to a standard representation.  
\end{enumerate}
\end{lemma}


Next, we show that we can find a small base and a small generating set for certain groups $H$ in quasipolynomial time. 
For a group $G \leq \Sy_n$, let $d(G)$ and $b(G)$ denote the size of the smallest generating set and base of $G$, respectively.

\begin{lemma} \label{get small gen set and base for $H$}
Let $H = \langle X \rangle \leq \Sy_n$ be given. 
\begin{enumerate}
    \item \label{small gen set in quasipoly time} If $|H| \leq n^{1+ \floor{\log{n}}}$, then in time $2^{O(\log^3{n})}$, we can decide if $d(H) \leq \ceil{\log{n}}$, and if so output such a generating set.
    \item \label{small base in quasipoly time} In time $2^{O(\log^2{n})}$, we can decide if $b(H) \leq \ceil{\log{n}} + 1$, and if so output such a base. 
    \item \label{if norm small prim then gent small gen set and small base in quasipoly time} If $N=N_{\Sy_n}(H)$ is a small primitive group, then $d(H) \leq \ceil{\log{n}}$ and $b(H) \leq \ceil{\log{n}} + 1$.
\end{enumerate}
\end{lemma}

\begin{proof}
\Cref{small gen set in quasipoly time}: We consider all $(\ceil{\log{n}})$-tuples $Z$ of elements of $H$ and decide for each $Z$ if $\langle Z \rangle = H$. 
The number of such tuples is $|H|^{\ceil{\log{n}}} \in 2^{O(\log^3{n})}$. 
By \Cref{poly time lib}.\ref{membership testing}, for each such tuple $Z$, we can decide if $\langle Z \rangle = H$ in polynomial time. \\
\Cref{small base in quasipoly time}: We consider all $(\ceil{\log{n}} + 1)$-tuples $B$ over $\{1,2, \ldots, n\}$. 
For each such $B$, we check if $B$ is a base of $H$ by checking if $H_{(B)}=1$, which can be done in polynomial time by \Cref{poly time lib}.\ref{compute pt stab}. 
Since there are $n^{\ceil{\log{n}} + 1} \in 2^{O(\log^2{n})}$ tuples to consider, the result follows. \\
\Cref{if norm small prim then gent small gen set and small base in quasipoly time}:  If $N$ is a small primitive group, then $H$ has order at most $n^{1+ \floor{\log{n}}}$. 
Since $H$ is a normal subgroup of a primitive group, by \cite[Theorem~1.1]{minGenOfPermAndtMatGps}, $d(H) \leq \log{n}$ or $H=\Sy_3$, so $d(H) \leq \ceil{\log{n}}$. By \cite{mariapiaColva}, $b(H) \leq b(N) \leq \ceil{\log{n}} + 1$. 
\end{proof}

Lastly we observe that the normaliser problem for groups of order less than $n^{1+ \floor{\log{n}}}$ can be solved in quasipolynomial time, even if they are not primitive.

\begin{proposition}\label{small group's norm in quasipoly time}
Let $H = \langle X \rangle \leq \Sy_n$ be given. If $|H| < n^{1+ \floor{\log{n}}}$, then $N_{\Sy_n}(H)$, and hence $N_K(H)$, can be computed in time $2^{O(\log^5{n})}$. 
\end{proposition}
\begin{proof}
By \Cref{poly time lib}.\ref{get a non-redundant base}, in polynomial time, we can check that $|H| < n^{1+ \floor{\log{n}}}$, compute an irredundant base $B$ for $H$ and remove from $X = \{x_1, x_2, \ldots, x_s \}$ the generators $x_i$ where $x_i \in \langle x_1, \ldots, x_{i-1} \rangle$. 
This gives a base $B$ and a generating set $Z$ for $H$ of size at most $\log{|H|} \in O(\log^2{n})$.

In \cite[proof of Theorem~3.3]{normOfPrim}, it is shown that in time $2^{O(|Z||B| \log{n})}$, we can construct a set containing all $|Z|$-tuples of elements of $H$ that are images of $Z$ under conjugation by elements of $N_{\Sy_n}(H)$. 
By \Cref{poly time lib}.\ref{check if map is homom} and \cite[Lemma~3.5]{luks2011polynomial}, for each such potential image, we can determine a conjugating element $\sigma \in N_{\Sy_n}(H)$ or show that no such $\sigma$ exists in polynomial time. 
\end{proof}


\section{Ample groups}
\label{section: ample}

In this section, we will introduce ample groups, and show that if $N=N_{\Sy_n}(H)$ is a large primitive group, then $H$ is ample. We then show that in quasipolynomial time, we can decide if a given group is ample and if so compute its normaliser. 
Finally, we will prove \Cref{main theorem}.  

\begin{definition} \label{def: ample group}
We define a subgroup $H$ of $\Sy_n$ to be \emph{ample} if there exist $m \geq 5$, $1 \leq k < m/2$ and $l\geq1$ 
such that $\soc{H}$ is permutation isomorphic to $\Amkl$. 
\end{definition}

Notice that an ample group can be imprimitive.

\begin{lemma}\label{socles of H and G} \label{soc G equal soc H} \label{G large prim then H large}
Let $H$ be a subgroup of $\Sy_n$ such that $N=N_{\Sy_n}(H)$ is a large primitive group. Then  $\soc{N}=\soc{H}$, and $H$ is ample. 
\end{lemma}

\begin{proof} 
We first show that $\soc{N}=\soc{H}$.  
The group $\soc{H}$ is characteristic in $H$, so $\soc{H} \trianglelefteq N$.
A large primitive group is either almost simple or of type PA, and so $N$ has a unique minimal normal subgroup (see \cite{ONanScottForPrim}). 
Therefore \[ 
\soc{N} \leq \soc{H} \leq H, \text{ so } \soc{N} \trianglelefteq H. 
\]
To see that $\soc{H} \leq \soc{N}$, let $M$ be a minimal normal subgroup of $H$. 
Then either $M \leq \soc{N}$ or $M \cap \soc{N}=1$. If $M \cap \soc{N}=1$, then $M \leq C_N(\soc{N})$. But by \cite[Theorem~4.3B]{dixonMortimer}, $C_N(\soc{N})=1$, a contradiction. Therefore, all minimal normal subgroups of $H$ are contained in $\soc{N}$, hence $\soc{H} \leq \soc{N}$, and so $\soc{N} = \soc{H}$.  

The largeness of $N$ implies that 
there exist $m,k$ and $l$ such that $\soc{N}$ is permutation isomorphic to $\Amkl$. Now since $\soc{H} = \soc{N}$, the group $H$ is ample. 
\end{proof}

The following is well known (see \cite[Theorem~4.5A]{dixonMortimer}, for example).  

\begin{lemma} \label{norm of Amkl is wreath}
Let $W \leq \Sym(\ksubsets{k}{m}^l)$ be $\sym{m}{k} \wr \Sy_l$ acting on $\ksubsets{k}{m}^l$,  for some $m \geq 5$, $1 \leq k < m/2$ and $l\geq1$.  
Then the normaliser in $\Sym(\ksubsets{k}{m}^l)$ of $\Amkl$ is $W$.
\end{lemma}

Next we give a polynomial time algorithm to decide whether $H$ is ample.

\begin{lemma}\label{socle is perm isom to Akmr in poly time} \label{can decide if norm is large in poly time}
Given $H = \langle X \rangle \leq \Sy_n$,  
in polynomial time, 
we can decide if $H$ is ample,
and if so output a permutation isomorphism from $\soc{H}$ to $\Amkl$, for some $m,k$ and $l$. 
\end{lemma}

\begin{proof}
By \Cref{poly time lib}.\ref{compute socle}, we can compute a generating set for $S:= \soc{H}$ in polynomial time. 
The group $S$ is a direct product of simple groups, so we can decide whether $S \cong \Al_m^l$, for some $m \geq5$ and $l \geq 1$, by checking if $S$ has $l$ composition factors, each isomorphic to $\Al_m$. 
By \Cref{poly time lib}.\ref{compute compo series}, this can be done in polynomial time. 

If $S$ is isomorphic to $\Al_m^l$, we next decide if there exists a $k$ such that $1 \leq k <m/2$ and $n= {\binom{m}{k}}^l$. If so, we construct an isomorphism $\iota:S \rightarrow \Amkl$ as follows. Initialise $N_1 = S$, then for $2 \leq i \leq l$, we iteratively find a minimal normal subgroup $M_i$ of $N_i$ and take $N_{i+1} = C_{N_i}(M_i)$, in polynomial time by \Cref{poly time lib}.\ref{find a minimal normal subgroup}. Then $M_i \cong \Al_m$ and $N_i = M_i \times C_{N_i}(M_i)$ for each $i$ and so $S = M_1 \times M_2 \times \ldots \times M_{l}$. 
We construct an isomorphism $\iota: S \rightarrow \Amkl$ using an isomorphism from each $M_i$ to a direct factor of $\Amkl$, via isomorphisms to a standard copy of $\Al_m$, in polynomial time by \Cref{poly time lib}.\ref{decide if simple}.

It remains to show how to find a permutation isomorphism between $S$ and $\Amkl$. 
Let $\Delta = \ksubsets{k}{m}^l$ and let $W \leq \Sym(\Delta)$ be as in \Cref{norm of Amkl is wreath}. 
If $m=6$ then there exists an involution $\alpha$ such that $\Aut(\alt{m}{k}) = \langle \sym{m}{k}, \alpha \rangle$, and we can obtain such an $\alpha$ in polynomial time by \Cref{poly time lib}.\ref{decide if simple}. 
So in polynomial time, we can write down all $2^l \leq 2^{\log{n}} = n$ coset representatives of $W$ in $\Aut(\Amkl)$.
We check if $S$ and $\Amkl$ are permutation isomorphic by checking if there exist such a coset representative $\lambda$ and a bijection 
$\sigma: \{1,2, \ldots, n\} \rightarrow \Delta $ such that $(\iota \lambda,\sigma)$ is a permutation isomorphism, in polynomial time by \cite[Lemma~2.7]{normOfPrim}. 
If $m\neq 6$ then $\Aut(\alt{m}{k}) = \sym{m}{k}$ and so $\Aut(\Amkl) =W$, and we may set $\lambda=1$. 
\end{proof}

If $H$ is ample and $l=1$, then $H$ is almost simple. 
The next result considers the case where $H$ is ample and $l>1$. 

\begin{theorem}\label{if normaliser is large prim then it can be computed in quasipolynomial time} 
Given $H = \langle X \rangle \leq \Sy_n$, we can decide if $H$ is ample and not almost simple, and if so compute $N=N_{\Sy_n}(H)$, in time $2^{O(\log{n} \log{\log{n}})}$.
\end{theorem}

\begin{proof}
By \Cref{can decide if norm is large in poly time}, in polynomial time, we can check if $H$ is ample and not almost simple, and if so obtain 
a permutation isomorphism $(\phi, \sigma)$ from $\soc{H}$ to $\Amkl$.  

We first show that we can compute a generating set for $\normOfSoc = N_{\Sy_n}(\soc{H})$ of size at most four in polynomial time. 
Let $W$ be as in \Cref{norm of Amkl is wreath}. 
The bijection $\sigma^{-1}$ induces an isomorphism from $\Sym(\ksubsets{k}{m}^l)$ to $\Sy_n$ that maps $W$ to $\normOfSoc$. 
By \cite[Lemma~4.3]{normOfPrim}, we can write down a generating set $Z$ for $W$ of size four, so $\normOfSoc = \langle \sigma^{-1}(Z) \rangle$ can be computed in polynomial time by \Cref{poly time lib}.\ref{compute im and preim of homom}. 

Next, since $\normOfSoc$ is isomorphic to $W$, 
\[
[\normOfSoc:\soc{H}] \leq 2^l |\Sy_l| \leq 2^l l^l = 2^{l + l \log{l}}. 
\]
As $l \leq \log{n}$, it follows that $[\normOfSoc:\soc{H}]  \leq  2^{\log{n}+ \log{n} \log{\log{n}}}$. 
Notice that $\soc{H} \trianglelefteq H$, so $H \leq  M$. 
Therefore $[\normOfSoc: H] \leq [\normOfSoc: \soc{H}] \leq 2^{2\log{n} \log{\log{n}}}$. 

By \cite[Lemma~4.4]{normOfPrim}, $N$ can be computed in time $O(n^3 [\normOfSoc: H]^2) = 2^{O(\log{n}\log{\log{n}})}$ (it is assumed in \cite{normOfPrim} that $H$ is primitive, but the assumption is not needed in the proof).  
\end{proof}


We end by giving the proof of \Cref{main theorem}. 

\begin{proof}[Proof of \Cref{main theorem}]
We first prove \Cref{decide if norm in prim}. Without loss of generality, suppose that $H$ is non-trivial. 
We check if $H$ is transitive, in polynomial time by \Cref{poly time lib}.\ref{compute orbits}. If not, then $N=N_{\Sy_n}(H)$ is not primitive, and we return false. 

Otherwise, by \cite[Lemma~4.5]{normOfPrim}, in polynomial time, we can check if $H$ is almost simple and if so compute $N$. Assume from now that $H$ (and hence $N$) is not almost simple.
Next, by \Cref{if normaliser is large prim then it can be computed in quasipolynomial time}, in time $2^{O(\log{n} \log{\log{n}})}$, we can determine if $H$ is ample and if so, compute $N$.  

If $H$ is not ample, then by \Cref{G large prim then H large}, $N$ is not large. 
So by \Cref{big small and sporadic primitive groups}, $N$ is primitive if and only if  $N$ is small. 
We check if $|H| \leq n^{1+ \floor{\log{n}}}$ in polynomial time by \Cref{poly time lib}.\ref{compute order} and return false if not. Next we look for a generating set of size at most $\ceil{\log{n}}$ and a base of size at most $\ceil{\log{n}} + 1$ for $H$ in time $2^{O(\log^3{n})}$ by \Cref{get small gen set and base for $H$}.\ref{small gen set in quasipoly time}--\ref{small base in quasipoly time}. 
If no such base and generating set exist, then by \Cref{get small gen set and base for $H$}.\ref{if norm small prim then gent small gen set and small base in quasipoly time}, $N$ is not primitive, and we return false. Otherwise we compute $N$ in time $2^{O(\log^3{n})}$ by \cite[Theorem~3.3]{normOfPrim}, and check if $N$ is primitive, in polynomial time by \Cref{poly time lib}.\ref{check if primitive}. 

\Cref{if norm prim then compute norm} now follows from the fact that, given $N$, the group $N_K(H) = K \cap N$ can be computed in time $2^{O(\log^3{n})}$, by Babai and Helfgott's results \cite{babaiGI, helfgottGI, luksHierarchy}.  
\end{proof}

\section*{Acknowledgements}
The first author is supported by a Royal Society grant (RGF\textbackslash EA\textbackslash181005).

\bibliographystyle{abbrv}

\end{document}